\documentclass[a4paper, 10pt]{article}
\usepackage{latexsym,amsthm,amssymb,amscd}
\usepackage{color,graphicx}
\author{{{\bf M. Boostan, S. Golalizadeh and }}  {\bf N. Soltankhah \thanks{Corresponding author: soltan@alzahra.ac.ir, soltankhah.n@gmail.com }} \\
{\footnotesize  {\bf Department of Mathematical Sciences}}\\  {\footnotesize {\bf Alzahra University}}\\ {\footnotesize {\bf P.O. Box 19834}}\\ {\footnotesize {\bf Tehran, I.R. Iran.}}}
\title {Super-simple $(v,4,2)$ directed designs and a lower bound for the minimum size of  their defining set }
\date{}
\newtheorem{main theorem}{Main Theorem}

\newtheorem{proposition}{Proposition}
\newtheorem{remark}{Remark}

\newtheorem{lemma}{Lemma}
\newtheorem{corollary}{Corollary}

\newtheorem{construction}{Construction}

\newcommand{\B}{\mathcal{B}}
\newcommand{\G}{\mathcal{G}}
\newcommand{\D}{\mathcal{D}}

\begin{document}
\maketitle
\begin{abstract}
In this paper, we show that for all $v\equiv 1$ (mod 3), there exists
a super-simple $(v,4,2)$ directed design. Also, we show that for these parameters there exists a super-simple $(v,4,2)$ directed design whose each defining set  has at least a half of the  blocks.
\end{abstract}
\section{Introduction and preliminaries}
A $t-(v,k, \lambda)$ design with parameters $v,k,\lambda$ is a pair $(X,\B)$ where $X$ is a set of $v$ elements,
 and $\B$ is a family of $k-$subsets of distinct elements of $X$, called blocks, with the property that every t-tuple of distinct elements
  occurs in exactly $\lambda$ blocks. The problem of evaluation $v$ such that there exists a $t-(v,k,\lambda)$ design for a specific $k,\lambda$
  is one of the most important problems in combinatorics.  \\

Some generalizations has been introduced for the concept of designs. Gronau and
Mullin \cite{gronau} for the first time, introduced a new definition of block
designs called super-simple block designs. A super-simple $t-(v,k,\lambda)$  design  is a block design such that any two blocks of the design
intersect in at most $t$ points. A simple block design is a block design such that it has no repeated blocks.
The existence of super-simple $(v,4,\lambda)$ designs have been
characterized for $2\leq \lambda \leq 9$ except $\lambda =7$, see \cite{chen3, chen6, chen1, 
chen5, chen0, gronau, zhang}. Also, the existence of super-simple
$(v,5,\lambda)$ designs have been characterized for $2\leq \lambda \leq 5$,   see  \cite{chen11, chen2, chen4, hans}.\\

A  $t-(v,k,\lambda)$ directed design is  a  pair $(X,\B)$, where $X$  is a $v$-set, $\B$ is a collection of ordered k-tuples of distinct elements of $X$ (called blocks) such that each ordered t-tuples of distinct elements of $X$ appears in exactly $\lambda$ blocks. By $t-(v,k,\lambda)$DD, we mean a $t-(v,k,\lambda)$ directed design. A  $t-(v,k,\lambda)$DD is super-simple if its underlying $t-(v,k,(t!)\lambda)$ design is super-simple.\\

In the rest of this paper, we use various types of combinatorial objects such as  trade, defining set, group divisible design(GDD), directed group divisible design(DGDD) and pairwise balanced design(PBD), that we review them here.\\

A set of blocks which is a subset of a unique $ t-( v, k,\lambda)$DD is said to be a defining
set of the directed design.
A minimal defining set is a defining set, no proper subset of which is a defining
set. A smallest defining set, is a defining set with the
smallest cardinality.
 
A $ (v,k,t) $ directed trade of volume  $s$ consists of two disjoint collections 
$ T_{1} $ and $ T_{2} $, each of  $s$ blocks, such that every ordered t-tuple of distinct elements of $ V $ 
is covered by precisely the same number of blocks of $ T_{1} $ as of $ T_{2} $.
Such a directed trade is usually denoted by $ T=T_{1}-T_{2} $. 
Blocks in $ T_{1} (T_{2}) $ are called the positive (respectively, negative)
 blocks of $ T $. In a $(v,k,t)$ directed trade, both
collections of blocks cover the same set of elements. This set of elements is called
the foundation of the trade. In \cite{soltan}, it has been shown that the minimum volume of a $(v,k,t)$ directed trade is $2^{\lfloor t/2\rfloor}$ and that directed trades of minimum volume and minimum foundation exist. In some parts of this paper, we handle with a special type of  directed trade, called a cyclical trade, defined as follows. Let $T=T_1-T_2$ be a $(v,k,t)$ directed trade of volume s with blocks $b_1,\cdots , b_s$ such that each pair of consecutive blocks of $T_1$ $(b_i,b_{i+1}, i=1,\cdots ,s$ (mod s)) is a trade of volume 2.\\
 
  If $  \D=(V,\B)  $ is a directed design, and if $ T_{1} \subseteq \B $,
  we say that $ \D $ contains the directed trade $ T $. For example the following super-simple $2-(13,4,2)DD $
$$\{(6,5,0,2), (0,1,6,4), (0,1,11,5), (1,0,3,9)\}$$  by (+1 mod  13)  
 contains the following directed trade:
\begin{center}
\begin{tabular}{cc}
$ T_{1}$  & $ T_{2}$ \\
 \hline (0,1,11,5) & (1,0,11,5) \\
 (1,0,3,9) & (0,1,3,9) \\
\end{tabular}
  
\end{center}

 Defining sets for directed designs are strongly related to trades. This relation is illustrated by the following result.
\begin{proposition}\cite{es.}
Let $ \D = (V,\B) $ be a $ t-(v,k,\lambda)DD$  and let $ S \subseteq \B $,
 then $ S $ is a defining set of $ \D $ if and only if $ S $ contains a block of every $ (v,k,t) $
 directed trade $ T=T_{1}-T_{2} $ such that $ T $ is contained in $ \D $.
\end{proposition}

Each defining set of a $ t-(v,k,\lambda)DD $, $\D$ contains at least one block in every trade in 
$ \D $. In particular, if $ \D $ contains $ m $ mutually disjoint directed trades then the smallest
defining set of $ \D $ must contain at least $ m $ blocks. If a directed design $ \D $ contains a cyclical trade of volume $s$, then each defining set for $ \D $ must contain at least $\lfloor \frac{s+1}{2} \rfloor$   blocks of $ T_{1} $.\\

The concept of directed trades and defining sets for directed designs
were investigated in articles \cite{es., soltan}.\\
 
A pairwise balanced design of order $ v $ with block sizes $ k \in  $K 
or PBD$ (v,K,\lambda) $ is a pair $ (V,\B) $, where $ V $ is a
$ v-$set, and $ \B $ is a collection of subsets (called blocks) of 
$V $ such that if $ B \in \B $ then $ \vert B \vert \in$K and every pair of distinct elements of 
$ V $ appears in precisely $ \lambda $ blocks. A
PBD$(v,K,1)$ is denoted  by PBD$(v,K)$.\\

A group divisible design of order $ v $ with block sizes $ k \in $K or 
$ (K,\lambda)-$GDD of type $g_1^{u_1}g_2^{u_2}...g_N^{u_N}$,
where $ u_{1},u_{2},...,u_{N} $ are non-negative integers, 
is a triple $ (V,\G,\B)$, where $ V $ is a $ v- $set that is partitioned into parts (called groups)
of sizes $ g_{1},g_{2},...,g_{N} $, and $ \B $ is a collection of subsets (called blocks)
of $ V $ such that if $ B \in \B $ then  $ \vert B \vert \in$K and every 
pair of distinct elements of $ V $ appears in precisely $ \lambda $ blocks or one
group but not in both. A $ (K,1)-$ GDD and $(\{k\},\lambda)-$GDD are denoted  by $K-$GDD and $(k,\lambda)-$GDD, respectively.\\

One can see deleting one point from a PBD$(v,K)$ gives a $K$-GDD of type 
 $g_1^{u_1}g_2^{u_2}...g_N^{u_N}$, where $ u_{1},u_{2},...,u_{N} $ 
are non-negative integers and for all $ i=1,2,...,N $, $ g_{i}=k_{i}-1; k_{i} \in K $.\\

A transversal design $TD(k,\lambda;n)$ is a $ (k,\lambda)-$GDD of type $ n^{k} $. When $\lambda=1$, we simply write $TD(k,n)$.\\

A directed group divisible design $ (K,\lambda)- $DGDD is a group divisible design  
in which every block is ordered and each ordered pair formed from distinct 
elements of different groups occurs is exactly $ \lambda $ blocks. A
$ (K,\lambda)- $DGDD of type $(g_{1}^{u_{1}},g_{2}^{u_{2}},...,g_{N}^{u_{N}})$  
 is super-simple if its underlying $ (K,2\lambda)- $GDD of type  
 $(g_{1}^{u_{1}},g_{2}^{u_{2}},...,g_{N}^{u_{N}})$  is super-simple.\\
 
Some results have been obtained on  $2-(v,k,\lambda)$DDs for special $k$ and $\lambda$  and their defining set. For example,  in \cite{es.}, has been proved that if $\D$ be a $2-(v,3,1)$DD, then a defining set of $\D$ has at least $\frac{v}{2}$ blocks. In \cite{Grannell}, it has been shown that for each admissible value of $v$, there exists  a simple $2-(v,3,1)$DD whose smallest defining sets have at least a half of the blocks. In
\cite{soltankhah1}, it has been shown that the necessary and
sufficient condition for the existence of a super-simple
$2-(v,4,1)$DD is $v\equiv 1$ (mod 3) and for these values of $v$ except $v=7$, there exists a super-simple $2-(v,4,1)$DD whose  smallest
defining sets have at least a half of the blocks. Also,
in \cite{soltankhah2}, it has been shown  that for all
$v\equiv 1,5$ (mod 10) except $v=5,15$, there exists a super-simple
$2-(v,5,1)$DD such that their smallest defining sets have at least a
half of the blocks.
In this paper, we prove that the necessary and sufficient condition
for the existence of a super-simple $2-(v,4,2)$DD is $v\equiv 1$ (mod
3) $(v\geq 10)$ and for these values of $v$, there exists  a super-simple
$2-(v,4,2)$DD  whose their smallest defining sets have at least a
half of the blocks. For this goal, we introduce the following quantity\\
$$d=\frac{the\  total \  number \ of \  blocks \   in \   a \   smallest  \   defining \  set \   in \   \D}{the \  total \   number \   of \  blocks \  in \  \D}$$
and we show for all admissible values of $v$, $d\geq \frac{1}{2}$.
\section{Recursive Constructions}
For some values of $v$, the existence of super-simple $(v,4,2)$DD will be proved by recursive constructions that we present them in this section for later use.
\begin{construction}(Weighting)\label{1}. Let $ (X,\G,\B) $ be a super-simple GDD with index $\lambda_1$, and let $ w:X\rightarrow Z^+ \bigcup \{0\} $ be a weight function on $X$, where $Z^+$ is the set of positive integers. Suppose that for each block $ B\in \B $, there exists a super-simple $ (k,\lambda_2)-GDD $ of type $ \{w(x): x\in B\} $. Then there exists a super-simple $(k,\lambda_1\lambda_2)-GDD$ of type $\{\sum_{x\in G_i} w(x):  G_i\in \G \}$.
\end{construction}
\begin{remark}
In the above construction, if the master GDD and input designs are directed and for all of them, we have $d\geq \frac{1}{2}$ then the resulted GDD is a super-simple  DGDD and for it, we have $d\geq \frac{1}{2}$.
\end{remark}
\begin{construction}
If there exists a super-simple $(k, \lambda)$-DGDD of type $g_1^{u_1}\cdots g_t^{u_t}$ with $d\geq \frac{1}{2}$
 and a super-simple
$(g_i +\eta, k,\lambda)$DD for each $ i (1\leq i\leq t)$ with $d\geq \frac{1}{2}$, then there exists a super-simple $(\sum_{i=1}^t  g_iu_i+\eta, k, \lambda)$DD with $d\geq \frac{1}{2}$, where
$\eta = 0 \ \  or  \ \ 1$.
\end{construction}
\begin{proof}
Replacing each $G_i\in \G$, $\vert G_i \vert =g_i$ with a super-simple $(g_i+\eta,k,\lambda)$DD gives a super-simple $(\sum_{i=1}^t  g_iu_i+\eta, k, \lambda)$DD. Obviously, if for master DGDD and for all input designs $d\geq \frac{1}{2}$, then for the resulted design, the inequality $d\geq \frac{1}{2}$ should be hold.
\end{proof}
\begin{construction}
If there exists a $K-GDD $ of type $g_1^{u_1}g_2^{u_2}...g_n^{u_n}$, a super-simple $(\alpha g_i +1,4,2)$DD for each i, $i=1,2,...,n$ and a super-simple $4$-DGDD of type $\alpha^k$ for each $k \in K$, then there exists a super-simple $(\alpha \sum_{i=1} ^n g_i u_i +1,4,2)$DD.
\end{construction}
\begin{proof} Let $(X,\G,\B)$ be a group divisible design with blocks of size $k\in K$ and groups of size $g_i$ for $i=1,2,\cdots, n$. Let $Y=X\times Z_{\alpha} \cup \{\infty\}$. Replacing each element $x\in X$ with $\alpha$ new points $\{x_1,x_2,\cdots, x_{\alpha}\}$ and each block $B\in \B$ of size $k\in K$ with a super-simple $4$-DGDD of type $\alpha^k$ such that its groups are $\{x_1,x_2,\cdots, x_{\alpha}\ : x\in X\}$ gives us a super-simple $4$-DGDD of type $(\alpha g_1)^{u_1}(\alpha g_2)^{u_2} \cdots (\alpha g_n)^{u_n}$. Finally, filling in the holes with a new point $\infty$ and using a  super-simple $2-(\alpha g_i+1,4,2)$DDs, we obtain a super-simple $2-(\alpha \sum_{i=1} ^n g_i u_i +1,4,2)$DD.
\end{proof}
\begin{remark}
In the above construction, if for master GDD and for all input designs, we have $d\geq \frac{1}{2}$, then for the resulting design, the inequality $d\geq \frac{1}{2}$ should be hold.
\end{remark}
\section{Direct Construction}
In this section, we construct some super-simple $2-(v,4,2)$DDs for some small admissible values of $v$ by direct construction and for these values of $v$, we show that the parameter $d$ for constructed designs is at least $\frac{1}{2}$. In the rest of paper, by $(v,4,2)$DD, we mean $2-(v,4,2)$DD.
\begin{lemma}\label{3} There exists a super-simple $(v,4,2)$DD for all $v\in \{10,13,16,19,22\\$,$25,28,31,34,40,43,58,67,79,94,103\}$,
whose their smallest defining sets have at least a half of the
blocks.
\end{lemma}
\begin{proof}
$ v=10 $: The following base blocks by (+1 mod 10)  form a super-simple $(10,4,2)$DD.\\
 
\begin{center}
\begin{tabular}{|c|c|}
  \hline
  (0,1,2,6) & (2,0,5,8) \\
  (1,0,4,3) &  \\
  \hline
\end{tabular}
\end{center}
This design has 30 blocks, the first column has 10 disjoint directed trades of volume 2
and the last column is a cyclical trade of volume 10.
 Since each defining set for this super-simple directed design must contain at least one 4-tuple of each directed trades in first column and five 4-tuples of cyclical trade in second column, then
 each defining set must contain at least 10+5=15 blocks. So for this super-simple $(10,4,2)$DD, we have $d\geq \frac{1}{2}$.\\

$v=13$: The following base blocks by (+1 mod 13)  form a super-simple $(13,4,2)$DD.\\
\begin{center}
\begin{tabular}{|c|c|}
  \hline
  (0,1,11,5) & (0,1,6,4) \\
  (1,0,3,9) & (6,5,0,2) \\
  \hline
\end{tabular}
\end{center}
 This design has 52 blocks,
each  of  two  columns  has $13$ disjoint directed trades of volume 2. Since  each defining set for this super-simple directed design must contain at least one 4-tuple of each directed trades, then 
 each defining set must contain at least $13+13=26$ blocks. So for this super-simple $(13,4,2)$DD, we have  $d\geq \frac{1}{2}$.\\
 
$v=16$: The following base blocks by (+1 mod 16) form a super-simple $(16,4,2)$DD.\\
\begin{center}
\begin{tabular}{|c|c|c|}
  \hline
  (0,1,6,8) & (4,1,0,10) & (0,1,9,5) \\
  (1,0,11,14) & (12,0,3,14) & \\
  \hline
\end{tabular}
\end{center}
 There are  80 blocks in a super-simple $(16,4,2)$DD. The first  two columns have 32 disjoint directed trades of volume 2, 
and the last column is a cyclical trade of volume 16.
Since each defining set for this super-simple directed design must contain at least one 4-tuple of each directed trade in the first two columns and eight 4-tuples of cyclical trade in the last column, then 
 each defining set must contain at least $16+16+8=40$ blocks. Therefore for this super-simple$(16,4,2)$DD the inequality $d\geq \frac{1}{2}$ is satisfied.\\

$v=19$: The following base blocks by (+1 mod 19) form a super-simple $(19,4,2)$DD.\\
\begin{center}
\begin{tabular}{|c|c|c|}
  \hline
  (0,6,1,3) & (4,1,0,12) & (0,1,8,10) \\
  (15,7,1,0) & (5,0,3,10) & (2,0,6,15) \\
  \hline
\end{tabular}
\end{center}
This design has $114$ blocks,
each  of three columns of this design has 19 disjoint directed trades of volume 2. Since each defining set for this super-simple  $(19,4,2)$DD must contain at least one 4-tuple of each directed trades, then 
 each defining set must contain at least $19+19+19=57$ blocks. Therefore for this super-simple $(19,4,2)$DD the inequality $d\geq \frac{1}{2}$ is satisfied.\\

$v=22$: The following base blocks by (+1 mod 22) form a super-simple $(22,4,2)$DD.\\
\begin{center}
\begin{tabular}{|c|c|c|c|}
  \hline
  (0,3,1,6) & (1,5,7,0) & (15,1,10,0) & (0,14,2,10) \\
  (0,4,1,16) & (2,0,13,9) & (8,0,19,3) & \\
  \hline
\end{tabular}
\end{center}
This design contains 154 blocks, each of first three  columns  has 22 
disjoint directed trades of volume 2 and the last column 
is a cyclical trade of volume 22. Since each defining set for this design
 must contain one 4-tuple of  each directed trade in the first three columns and 11 4-tuples of cyclical trade in the last column, then each
 defining set contains at least $22\times3+11 = 77$ blocks. So $d\geq \frac{1}{2}$.\\

$v=25$: The following base blocks by (+1 mod 25) form a super-simple $(25,4,2)$DD.\\

\begin{center}
\begin{tabular}{|c|c|c|c|}
  \hline
  (0,1,18,3) & (0,11,1,7) & (13,6,1,0) & (2,0,16,12) \\
  (16,0,8,23) & (14,9,3,0) & (16,21,0,4) & (19,0,22,24)\\
  \hline
\end{tabular}
\end{center}
This design contains 200 blocks,
each of four columns  has 25 disjoint directed trades of volume 2. Since each
defining set for this design must contain one 4-tuple of each directed trade, then each
 defining set contains at least $ 25\times4 =100$ blocks. So the desired inequality follows.\\
 
 $v=28$: The following base blocks by (+1 mod 28) form a super-simple $ (28,4,2)$DD.

\begin{center}
\begin{tabular}{|c|c|c|c|c|}
  \hline
  (4,0,2,1) & (1,20,0,26) & (0,3,20,7) & (0,23,3,14) & (0,18,2,12) \\
  (15,0,1,5) & (19,0,6,13) & (17,0,10,5) & (19,11,7,0) & \\
  \hline
\end{tabular}
\end{center}
This design has 252 blocks, each of  the first four columns 
 has 28 disjoint directed trades of volume 2 and the last column is a cyclical trade of volume 28. Since each defining set for this design
 must contain one 4-tuple of  each directed trade in the first four columns and 14 4-tuples of cyclical trade in the last column, then each 
 defining set contains  at least $ 28\times4+14=126 $ blocks. So the desired inequality follows.\\

$v=31$: The following base blocks by (+1 mod 31) form a super-simple $ (31,4,2)$DD.

\begin{center}
\begin{tabular}{|c|c|c|c|c|}
  \hline
  (0,3,8,1) & (11,5,0,1) & (7,1,19,0) & (0,2,23,9) & (15,0,4,3) \\
  (3,0,18,13) & (3,0,11,17) & (0,2,15,6) & (26,2,0,11) &(0,27,19,10) \\
  \hline
\end{tabular}

\end{center}
This design contains 310 blocks,
each of the five columns has 31 disjoint directed trades of volume 2. Since each
defining set for this design must contain one 4-tuple of each directed trade, then each defining set 
 contains at least $ 31\times5=155 $ blocks. So $d\geq \frac{1}{2}$. \\

$v=34$: The following base blocks by (+1 mod 34) form a super-simple $ (34,4,2)$DD.

\begin{center}
\begin{tabular}{|c|c|c|}
  \hline
  (25,0,4,11) & (2,6,0,21) & (0,1,9,3) \\
  (2,12,9,0) & (7,3,29,0) & (1,7,19,0)\\
  (20,0,2,5) &    &                 \\
  (14,0,1,24) & (21,0,11,29) & \\
  (5,1,0,17) & (28,20,11,0) & \\
  \hline
\end{tabular}

\end{center}
This design has 374 blocks, the first column contains 34 cyclical trades of volume 5 and the last two columns contain
 102 disjoint directed trades of volume 2. Since each 
defining set for this design must contain $\frac{5+1}{2}\times 34$ 4-tuples of cyclical trade and one 4-tuple of each directed trade  of volume 2, then
 each defining set contains at least $ 102+102=204 $ blocks. So the desired inequality follows.\\

$v=40$: The following base blocks by (+1 mod 40) form a super-simple $ (40,4,2)$DD.

\begin{center}
\begin{tabular}{|c|c|c|c|}
  \hline
  (4,1,0,2) & (0,4,11,32) & (0,16,7,29) & (26,20,0,5) \\
  (0,3,18,23) & (5,0,33,17) & (32,1,0,38) &               \\ 
    &               &                   &                 \\     
  (0,12,3,33) & (17,4,0,30) & (0,6,1,15) &                 \\
  (23,0,10,2) & (18,0,4,29) &  (22,30,6,0) &                 \\
  \hline
\end{tabular}

\end{center}

This design has 520 blocks, contains 240 disjoint directed trades of volume 2 in the first three columns and 
 a cyclical trade in the last column. 
Since each defining set for this design must contain one 4-tuple of each directed trades in the first three columns and 20 4-tuples of the cyclical trade, then each
 defining set contains at least $ 240+20=260$ blocks. So $d\geq \frac{1}{2}$.\\

$v=43$: The following base blocks by (+1 mod 43) form a super-simple $ (43,4,2)$DD.

\begin{center}
\begin{tabular}{|c|c|c|c|}
  \hline
  (0,1,8,3) & (0,1,4,10) & (1,19,0,12) & (12,0,42,29) \\
  (0,5,20,27) & (11,22,3,0) & (0,10,35,26) & (4,16,0,29)\\
 &                       &                           &                     \\
  (0,4,2,23) & (2,17,0,8) & (4,20,0,9) &                   \\
  (10,24,3,0) & (5,0,33,18) & (12,6,26,0) &                \\
  \hline
\end{tabular}

\end{center}
This design has 602 blocks, contains 301 disjoint directed trades of volume 2. Since each 
defining set for this design must contain one 4-tuple of each directed trades, then each
 defining set contains at least $301 $ blocks. So $d\geq \frac{1}{2}$.\\

$v=58$: The following base blocks by (+1 mod 58) form a super-simple $ (58,4,2)$DD.

\begin{center}
\begin{tabular}{|c|c|c|c|}
  \hline
  (0,1,56,3) & (1,31,0,27) & (0,35,2,24) & (0,12,3,16) \\
  (0,1,14,9) & (1,32,0,52) & (5,0,23,15) & (40,15,3,0) \\
 &                       &                           &                     \\
  (0,4,37,19) & (13,22,6,0) & (0,17,34,5) &  (15,26,0,7)  \\
  (30,13,4,0) & (20,34,0,6) & (27,17,0,7) &   (24,36,0,47) \\              
  &                       &                           &                     \\
 (22,38,0,8) & (0,25,6,35) &                   &    \\
 (0,56,40,19) &                 &                      &   \\
  \hline
\end{tabular}

\end{center}
This design has 1102 blocks, contains 522 disjoint directed trades of volume 2 and 
 a cyclical trade of volume 58. Then each
 defining set contains at least $522+29=551$ blocks. So the desired inequality follows.\\

$v=67$: The following base blocks by (+1 mod 67) form a super-simple $ (67,4,2)$DD.

\begin{center}
\begin{tabular}{|c|c|c|c|}
  \hline
  (1,10,0,16) & (7,0,33,30) & (42,0,3,7) & (14,38,0,2) \\
  (33,24,0,8) & (0,10,17,28) & (6,2,0,19) & (45,2,0,21) \\
 &                       &                           &                     \\
  (5,13,0,50) & (0,14,6,47) & (23,7,0,49) &  (9,19,31,0)  \\
  (5,8,0,2) & (44,29,14,0) & (0,38,54,5) &   (12,0,25,39) \\              
  &                       &                           &                     \\
 (1,0,30,12) & (20,0,9,40) & (21,0,4,36)       &    \\
 (11,0,1,46) & (0,20,41,1)  & (27,32,4,0)       &   \\
  \hline
\end{tabular}

\end{center}
This design has 1474 blocks, contains 737 disjoint directed trades of volume 2. Then each 
defining set for this design  contains at least  $ 737 $ blocks. So $d\geq \frac{1}{2}$. \\

$v=79$: The following base blocks by (+1 mod 79) form a super-simple $ (79,4,2)$DD.

\begin{center}
\begin{tabular}{|c|c|c|c|c|}
  \hline
  (17,0,37,1) & (0,1,22,15) & (0,31,2,23) & (1,33,0,9) & (35,0,19,2) \\
  (13,30,0,4) & (1,0,30,12) & (51,30,0,10) & (5,0,12,45) & (0,25,3,38)\\ 
  &                       &                           &                 &    \\
  (3,0,30,45) & (39,0,4,29) & (19,27,4,0) & (15,0,5,39) &
  (37,0,6,26) \\
  (43,0,3,31) & (4,0,51,37) & (27,0,13,5) & (11,6,28,0) &
  (6,0,54,16)\\ 
 &                       &                           &                     & \\
  (0,18,27,6) & (3,26,44,0) & (0,50,7,26) &      &
  \\
  (20,2,0,34) & (9,25,2,0) & (0,43,24,11) &      &          \\
  \hline
\end{tabular}

\end{center}

This design has 2054 blocks, contains 1027 disjoint directed trades of volume 2. Then each 
defining set for this design  contains at least  $ 1027 $ blocks.  So $d\geq \frac{1}{2}$. \\

$v=94$: The following base blocks by (+1 mod 94) form a super-simple $ (94,4,2)$DD.

\begin{center}
\begin{tabular}{|c|c|c|c|c|}
  \hline
  (1,23,0,40) & (2,0,20,43) & (20,9,0,82) & (0,79,88,49) & (8,0,61,29) \\
  (0,5,28,42) & (13,87,6,0) & (0,12,26,71) & (0,15,4,45) &
  (0,32,8,44)\\
  &                       &                           &                    &  \\
  (47,19,0,3) & (0,13,32,57) & (26,34,0,7) & (17,44,4,0) & (5,43,0,25) \\
  (18,46,0,4) & (0,69,91,52) & (16,1,0,34) & (3,56,29,0) &
  (0,56,2,31)\\ 
  &                       &                           &                    & \\
  (29,0,5,59) & (61,0,10,1) & (59,0,11,1) & (73,0,16,6) & (0,6,17,48) \\
  (36,5,0,51) & (2,70,37,0) & (12,4,84,0) & (0,58,24,3) &(38,7,0,16) \\ 
  &                       &                           &                  &   \\
  (30,0,49,2)  &     &    &       &      \\
  \hline
\end{tabular}
\end{center}
This design has 2914 blocks, contains 1410 disjoint directed trades of volume 2 and 
 a cyclical trade of volume 94. So each
defining set for this design  contains at least  $1410+47=1457 $ blocks. So the desired inequality follows. \\

$v=103$: The following base blocks by (+1 mod 103) form a super-simple $ (103,4,2)$DD.

\begin{center}
\begin{tabular}{|c|c|c|c|c|}
  \hline
  (0,1,15,31) & (21,1,0,50) & (51,0,2,25) & (34,53,2,0) & (12,39,0,3) \\
  (0,89,97,73) & (6,50,0,17) & (5,0,51,15) & (2,65,19,0) &
  (20,42,3,0)\\
  &                       &                           &                    &  \\
  (45,22,0,4) & (19,4,0,47) & (0,3,51,24) & (29,0,5,62) & (0,2,37,73) \\
  (0,6,29,38) & (0,4,31,65) & (49,1,0,46) & (5,0,53,40) &
  (70,43,8,0)\\ 
  &                       &                           &                    & \\
  (69,0,5,26) & (19,12,59,31) & (11,37,23,0) & (0,1,93,8) & (0,16,72,44) \\
  (40,0,82,9) & (45,0,7,20) & (8,0,18,32) & (92,0,77,59) &(10,47,0,69) \\ 
  &                       &                           &                  &   \\
  (64,4,29,0)  & (12,48,61,0)    &    &       &      \\
  (0,66,41,86) & (22,13,0,6)   &             &             &       \\    
  \hline
\end{tabular}
\end{center}
This design has 3502 blocks, contains 1751 disjoint directed trades of volume 2. So each defining set of this design contains at least 1751 blocks. Then the inequality $d\geq \frac{1}{2}$ is satisfied. Then the proof is complete.
\end{proof}
\section{Super-simple directed group divisible designs with block
size $4$ and index $2$}
In this section, we construct some super-simple DGDDs with $d\geq \frac{1}{2}$,  that we need them in our main result.
\begin{lemma}\label{4}
There exists a super-simple $(4,2)$-DGDD of type $ 3^t $  for  $ t \in \lbrace 6,7,8,9,$\\$13 \rbrace $ with $d\geq \frac{1}{2}$.
\end{lemma}
\begin{proof}
Let the point set be $ X=Z_{3t} $ and the group set be \\
 $\G = \lbrace\lbrace 0,t,2t \rbrace +i \,  \vert \,  \ 0 \leq i \leq t-1 \rbrace$. 
The base blocks are listed below. All the bellow base blocks are developed by mod 3t.\\ 
\begin{center}
$t=6$:
\begin{tabular}{|c|c|}
  \hline
  (2,0,5,9) & (0,1,2,16) \\
  (7,10,0,2) & (11,4,1,0) \\
  (1,5,0,10) &  \\
  \hline
\end{tabular}

\end{center}

This super-simple DGDD has 90 blocks, in the first column  
it contains 18 cyclical trades of 
volume 3 and in the second column 18 disjoint directed trades of volume 2. 
Then each defining set for this super-simple DGDD  contains at least  $ \frac{3+1}{2}\times 18+18=54 $ blocks.  So $d\geq \frac{1}{2}$.\\

\begin{center}
$t=7$:
\begin{tabular}{|c|c|c|}
  \hline
  (0,5,1,13) & (0,11,17,19) & (0,1,4,6) \\
  (11,10,1,5) & (12,11,0,19) & (7,12,4,1) \\
  \hline
\end{tabular}

\end{center} 

This super-simple DGDD has 126 blocks, each of  three columns 
has 21 disjoint directed trades of volume 2. Therefore
each defining set for this super-simple DGDD  contains at least  $ 21\times3=63 $ blocks. So $d\geq \frac{1}{2}$.\\

\begin{center}
$t=8$:  
\begin{tabular}{|c|c|c|}
  \hline
  (0,1,13,6) & (0,4,10,15) & (12,7,0,2) \\
  (2,1,0,4) & (21,0,4,18) & (2,0,9,15) \\
  (1,2,11,22) &   &   \\
  \hline
\end{tabular}

\end{center} 
This super-simple DGDD contains 168 blocks, the first column 
contains 24 cyclical trades of volume 3 and each of two other  columns has 24 
disjoint directed trades of volume 2. Then each
 defining set for this super-simple DGDD  contains at least  $ \frac{3+1}{2}\times 24+2\times 24=96 $ blocks. So the desired inequality follows.\\

\begin{center}
$t=9$:
\begin{tabular}{|c|c|c|}
  \hline
  (0,6,1,13) & (1,0,2,5) & (0,6,2,17) \\
  (3,13,1,23) & (3,11,0,24) & (0,14,8,4) \\
  (17,1,6,4) & (1,8,0,20) &    \\
  \hline
\end{tabular}

\end{center} 
This super-simple DGDD has 216 blocks, 
each of two first column  contains 27  cyclical trades of volume 3 and the last column  contains 27  disjoint directed trades of volume 2.
So each defining set for this super-simple DGDD  contains 
at least $ 2(\frac{3+1}{2}\times 27)+27=135 $ blocks. So the desired inequality follows.\\

$t=13$:

\begin{flushleft}

\begin{tabular}{|c|c|c|c|c|}
\hline ( 3,0,12,21 )  & ( 2,1,0,4 )  &( 8,0,2,19 )   & ( 0,16,22,5 )  & ( 5,0,11,1 )  \\ 
 ( 19,0,3,35 ) & ( 4,0,32,18 )  & ( 2,0,29,14 )  &( 0,17,1,8 )   & ( 0,5,15,36 )  \\ 
 (5,12,0,20) & ( 14,0,33,24 )  &  &  &  \\ 
\hline 
\end{tabular}
\end{flushleft} 

This super-simple DGDD has 468 blocks,  each of two first  columns of
this super-simple DGDD contains 39 cyclical trades of 
volume 3 and each of the next three columns has 39 disjoint directed trades of volume 2.
So each defining set for this super-simple DGDD has  at least $ 2(\frac{3+1}{2}\times 39)+3\times 39=273$ blocks. Therefore $d\geq \frac{1}{2}$. 
\end{proof}
\begin{lemma}\label{5}
There exists a super-simple $(4,2)$-DGDD of type $ t^4 $  for
  $ t \in \lbrace 4,5,6,13,\\19,22 \rbrace $ with $d\geq \frac{1}{2}$.
\end{lemma}

\begin{proof}
Let $ X=Z_{4t} $, $ \G=( \{ 0,4,...,4(t-1)  \}+i \, \vert \,   \ 0 \leq i \leq 3 \} $ 
and the base blocks are listed below, all the bellow base blocks are developed by mod 4t.\\

\begin{center}
$t=4$:
\begin{tabular}{|c|c|}
  \hline
  (0,1,3,10) & (0,5,2,11) \\
  (2,0,3,13) & (0,15,14,5) \\
  \hline
\end{tabular}

\end{center}

This super-simple DGDD has 64 blocks, each of  two columns 
has 16 disjoint directed trades of volume 2.
So each defining set for this super-simple DGDD must contain at least  $16\times2=32 $ blocks. Then the desired inequality follows.\\

\begin{center}
$t=5$:
\begin{tabular}{|c|c|}
  \hline
  (1,0,10,3) & (0,7,2,5) \\
  (0,1,18,11) & (3,0,14,9) \\
  (1,6,0,7) &   \\
  \hline
\end{tabular}

\end{center}
This super-simple DGDD has 100 blocks, the first column of
this super-simple DGDD contains  20 cyclical trades of 
volume 3 and the second column has 20 disjoint directed trades of volume 2, 
so each defining set contains at least $ \frac{3+1}{2}\times 20+20=60 $  blocks. So the desired inequality follows.\\

\begin{center}
$t=6$:
\begin{tabular}{|c|c|}
  \hline
  (0,1,2,7) & (2,0,5,19) \\
  (1,0,14,11) & (0,2,15,21) \\
  (14,0,23,17) & (0,15,22,9) \\
  \hline
\end{tabular}

\end{center}
This super-simple DGDD has 144  blocks, each of  two columns of
this super-simple DGDD contains 24 cyclical trades of 
volume 3, so each defining set for this super-simple DGDD 
 contains at least $2( \frac{3+1}{2}\times 24)=96$ blocks. So the desired inequality follows.\\

\begin{flushleft}
$ t=13 $:
\begin{tabular}{|c|c|c|c|c|}
  \hline
  (0,10,15,1) & (1,7,0,2) & (3,0,25,6) & (0,5,39,26) & (1,0,18,27) \\
  (0,29,19,2) & (0,13,7,30) & (7,10,0,21) &   &   \\
  (0,29,43,18) & (5,15,22,0) &   &  &  \\
  \hline
\end{tabular}

\end{flushleft}
This super-simple DGDD has 520 blocks, each of   two first columns  of
this super-simple DGDD contains 52 cyclical trades of 
volume 3, the third column has 52 disjoint directed trades of volume 2,
and  two last  columns are cyclical trades of volume 52. 
So each defining set for this super-simple DGDD  contains  at least $ 2(\frac{3+1}{2}\times 52)+52+2 \times 26=312 $ blocks. Then the inequality $d \geq \frac{1}{2}$ follows.\\

\begin{flushleft}
$ t=19 $:
\begin{tabular}{|c|c|c|c|c|}
  \hline
  (30,0,3,9) & (22,0,35,1) & (22,11,0,29) & (0,47,33,70) & (6,25,0,51)  \\
  (0,2,13,43) & (5,15,0,22) & (0,11,1,26) & (9,26,55,0) & (0,57,15,34) \\
  (2,0,33,7) & (0,11,1,26) & (7,17,54,0) & ( 5,0,14,39 )   &  \\
        &        &           &       &    (10,1,0,3)  \\
(1,19,0,6) & (3,33,0,62) & (2,23,0,37) & (13,0,31,58) &  (18,3,0,41) \\
  \hline
\end{tabular}

\end{flushleft}
This super-simple DGDD has 1520 blocks, each of four first columns contains 76 cyclical trades of volume 3 and one cyclical trade of volume 76, and the last column has 152 disjoint directed trades of volume 2. 
So each defining set for this super-simple DGDD  contains  at least $ 4(\frac{3+1}{2} \times 76)+4 \times 38+152=912$ blocks. Then the inequality $d \geq \frac{1}{2}$ follows.\\

$ t=22 $:
\begin{flushleft}

\begin{tabular}{|c|c|c|c|c|}
  \hline
 (0,6,19,1)  &(2,27,5,0)   & (2,39,0,25)  & (35,0,2,29)  & (0,30,7,69)  \\ 
  (6,0,15,33) &(39,0,5,26)   & (3,18,37,0)  & (12,27,0,10)  & (0,13,39,46)  \\ 
  (15,0,62,37)  & (25,0,14,3)  & (9,19,0,54)  &(10,0,3,41)   & (7,0,65,50)  \\ 
    &  &  &  &  \\ 
 (0,57,74,79)  & (42,0,11,53)  &(38,1,0,67)  & (1,0,43,2)  &  \\ 
(17,0,35,58)   &  (21,10,0,55) &(0,29,38,59)  &  &  \\ 
\hline 
\end{tabular} 

\end{flushleft}
This super-simple DGDD has 1936 blocks,  
each of three first  columns contains 88 cyclical trades of volume 3 and 88 disjoint directed trades of volume 2, the fourth column contains 88 cyclical trades of volume 3 and a cyclical trade of volume  88 and the last column contains 88 cyclical trades of volume 3. 
So each defining set for this super-simple DGDD  contains  at least $ 5(\frac{3+1}{2}\times 88)+3\times 88+44=1188 $ blocks. Then the desired inequality follows.
\end{proof}
\begin{lemma}\label{6}
There exists a super-simple $(4,2)$-DGDD of type $ 9^t $  for  
$ t \in \lbrace 4,5 \rbrace $ with $d\geq \frac{1}{2}$.
\end{lemma}

\begin{proof}
Let the point set be $ X=Z_{9t} $ and the group set be \\ 
$\G = \lbrace\lbrace 0,t,2t,...,8t \rbrace +i \, \vert \,  \  0 \leq i \leq t-1 \rbrace$. 
The base blocks are listed below. All the bellow base blocks are developed by mod 9t.\\ 

\begin{center}
$ t=4 $:
\begin{tabular}{|c|c|c|c|}
  \hline
  (1,0,2,7) & (0,1,10,19) & (3,0,21,10) & (2,0,5,31) \\
  (3,14,0,17) & (1,0,23,14) & (9,0,15,2) & (10,0,31,25) \\
  \hline
\end{tabular}

\end{center}
This super-simple DGDD has 288 blocks, each of four columns of this DGDD
has 36 disjoint directed trades of volume 2.
So each defining set for this super-simple DGDD  contains  at least $ 36\times4=144 $ blocks. Therefore $d\geq \frac{1}{2}$.\\

\begin{center}
$ t=5 $:
\begin{tabular}{|c|c|c|c|c|}
  \hline
  (0,1,2,4) & (2,0,11,19) & (6,13,27,0) & (11,0,42,33) & (0,26,3,19) \\
  (1,0,8,14) & (4,0,37,13) & (1,7,0,24) & (0,16,27,39) &  (17,3,0,29)\\
  (16,0,43,34) & (0,4,41,28) &   &   &    \\
  \hline
\end{tabular}

\end{center}
This super-simple DGDD has 540 blocks, each of two first  columns of
this super-simple DGDD contains 45 cyclical trades of 
volume 3 and each of the next three columns has 45 disjoint directed trades of volume 2. 
So each defining set for this super-simple DGDD has  at least $ 2(\frac{3+1}{2}\times 45)+3\times 45=315$ blocks. Therefore $d\geq \frac{1}{2}$.
\end{proof}
 \begin{lemma}\label{8}
There exists a super-simple $(4,2)$-DGDD of type $ 6^{5} $ with $d\geq \frac{1}{2}$.
\end{lemma}

\begin{proof}
Let $ X=Z_{30} $, $ \G=\{\{ 0,5,10,15,20,25 \}+i \, | \,  \  0 \leq i \leq 4 \} $, the set base blocks 
are listed below. They are developed by mod $ 30 $.

\begin{center}
\begin{tabular}{|c|c|c|c|}
\hline (7,9,0,1) & (19,1,12,0) & (1,2,4,0) & (3,9,0,17) \\ 
(3,0,7,16)  &(6,13,0,19) & (4,8,16,0) &  \\ 
 (2,0,18,21) &  &  &  \\  
\hline 
\end{tabular}

\end{center}
This super-simple DGDD has 240 blocks, the first column contains 30  cyclical trades of 
volume 3 and each of the second and third column has 30 disjoint directed trades of volume 2 and  
the last column is a cyclical trade of volume 30. 
So each defining set for this super-simple DGDD  contains  at least $\frac{3+1}{2}\times 30+ 2\times 30+15=135$ blocks. Then the desired inequality follows.
\end{proof}
\section{Proof of Main Theorem}
This section is devoted to find super-simple $(v,4,2)$DDs for some admissible  values of $v$ by  recursive constructions presented in Section 2 and using super-simple DGDDs obtained in  Section 4. First we present some lemmas which  are generalized form of the lemmas in \cite{chen1} to directed designs. Finally, we conclude the main result in the end of this section.

\begin{lemma}\label{11}
There exists a super-simple $(v,4,2)$DD for $ v \in \{ 37,46,52,76, 88 \} $ with $d\geq \frac{1}{2}$.
\end{lemma}

\begin{proof}
By Lemma \ref{6}, we know that  there exists a super-simple $ (4,2)-$DGDD of group type $ 9^{t} $ for 
$ t \in \{4,5\} $. Since $ \vert G \vert +1=9+1=10$ for each group $ G $ of the super-simple DGDD,
by Lemma \ref{3} and Construction 2, we conclude that there exists a super-simple $ (v,4,2)$DD that $ v\in \{37,46\}$ 
with $d\geq \frac{1}{2}$.\\

By Lemma \ref{5}, there exists a super-simple $ (4,2)-$DGDD of group type $ t^{4} $ for $ t \in \{13,19,22\} $  with $d\geq \frac{1}{2}$.  
Since $ \vert G \vert =13,19 $ or $22$, for each group $ G $ of the super-simple DGDD,
by Lemma \ref{3} and Construction 2 it follows there exists a super-simple $(v,4,2)$DD for $ v\in \{52,76,88 \}$ with
$d\geq \frac{1}{2}$.
\end{proof}
\begin{lemma}
There exists a super-simple $(v,4,2)$DD with
$d\geq \frac{1}{2}$ for $v\in \{49, 61, 73\\, 97, 121\}$.
\end{lemma}
\begin{proof}
By using a super-simple $ (4,2)-$DGDD of group type $ t^{4} $, 
$4 \leq t \leq 6 $  by Lemma \ref{5} with $d\geq \frac{1}{2}$ and
applying Construction \ref{1} by  using a $TD(4, 3)$ as an input design, we
obtain a super-simple $ (4,2)-$DGDD of group type $(3t)^{4}$ with
$d\geq \frac{1}{2}$. On the other hand, for $4\leq t\leq 6$ there exists a super-simple $(3t+1,4,2)$DD  with
$d\geq \frac{1}{2}$. So by Construction 2, there exists a super-simple $(12t+1,4,2)$DD with
$d\geq \frac{1}{2}$ for $4\leq t\leq 6$.\\

By considering  a super-simple  $(4, 2)-$DGDD of group type $6^4$ with $d \geq \frac{1}{2}$,
 obtained in Lemma 3 
 and applying Construction 1 by  using a $TD(4, m)$ in which
$m \in \{4, 5\}$ as an input design, we obtain a super-simple $(4, 2)-$DGDD of group
type ${(6m)}^4$ with
$d\geq \frac{1}{2}$. By considering the existence of super-simple $(6m+1,4,2)$DD from Lemma \ref{3} and using Construction 2, we conclude that there exists a  super-simple $(24m + 1, 4, 2)$DD with
$d\geq \frac{1}{2}$.
\end{proof}
\begin{lemma}
There exists a super-simple $(v,4,2)$DD for $ v \in \{64,100,112 \} $ with $d\geq \frac{1}{2}$.
\end{lemma}
\begin{proof}
By using  a
super-simple $(4, 2)-$DGDD of group type $4^4$ with $d\geq \frac{1}{2}$ from Lemma \ref{5} and applying
Construction 1 by  using a $TD(4, m)$ in which $m\in \{4, 7\}$ as an input design, we obtain
a super-simple $(4,2)-$DGDD of group type ${(4m)}^4$ with $d\geq \frac{1}{2}$. By Lemma \ref{3}, there exists a
super-simple $(4m, 4, 2)$DD with $d\geq \frac{1}{2}$ for $m\in \{4, 7\}$ . Then by Construction 2, we obtain a super-simple $(16m, 4, 2)$DD.\\

By considering  a super-simple $(4,2)-$DGDD of group type $5^4$ with $d\geq \frac{1}{2}$ from Lemma \ref{5} and applying Construction 1 by using a $TD(4, 5)$ as an
input design, we obtain a super-simple $(4, 2)-$DGDD of group type ${(25)}^4$ with $d\geq \frac{1}{2}$. Since by Lemma \ref{3} there exists
a super-simple $(25, 4, 2)$DD with $d\geq \frac{1}{2}$, then by Construction 2, we obtain a super-simple $(100, 4, 2)$DD with $d\geq \frac{1}{2}$.
\end{proof}
\begin{lemma}
There exists a super-simple $(v,4,2)$DD for $ v \in \{70,127 \} $ with $d\geq \frac{1}{2}$.
\end{lemma}

\begin{proof}
We have a  $ 4- $GDD of group type $ 2^{7} $ with  set 
$ X=Z_{14} $, group set  $ \G=\{ \{0,7\}+i \, | \,  \ 0 \leq i \leq 6 \} $ and  block set $ \B = \{( 0,1,4,6 )  $ mod $ 14\}$.
Applying Constraction \ref{1} and using a super-simple $ (4,2)- $DGDD of group type $ 5^{4} $
from  Lemma \ref{5} with $d\geq \frac{1}{2}$ as an input design, we obtain a super-simple $ (4,2)-$DGDD of 
group type $ 10^{7} $ with $d\geq \frac{1}{2}$. By Lemma \ref{3} and Construction 2 we obtain a super-simple $(70,4,2)$DD with
 $d\geq \frac{1}{2}$.\\

By using a $4-$GDD of group type $2^{7}$, applying Construction \ref{1} and using
a super-simple $(4,2)-$DGDD of group type $ 9^{4} $ (see Lemma \ref{6}) with 
$d\geq \frac{1}{2}$ as an input design, we get
a super-simple $(4,2)-$DGDD of group type $ 18^{7} $ with $d\geq \frac{1}{2}$. Since $18 + 1 = 19$, by Lemma
\ref{3} and Construction 2,  we obtain a super-simple $(127,4,2)$DD with $d\geq \frac{1}{2}$.
\end{proof}
\begin{lemma}
There exists a super-simple $(v,4,2)$DD for $ v \in \{55, 82, 85, 106, 109,$\\$ 118\} $ with $d\geq \frac{1}{2}$.
\end{lemma}
\begin{proof}
By Lemma \ref{4}, we know that  there exists a super-simple
$(4, 2)-$DGDD of group type $3^t$ for $t \in\{6, 7, 9, 13\}$ with $d\geq \frac{1}{2}$. Applying Construction 1 by using $TD(4, m)$ in which $m\in\{3,4\}$ as an input design we obtain a super-simple $(4,2)-$DGDD of type ${(3m)}^t$ with $d\geq \frac{1}{2}$. By considering the existence of super-simple $(3m+1,4,2)$DD and by Construction 2, we conclude that there exists a super-simple $(3mt+1,4,2)$DD where  $m\in\{4,3\}$ and $t \in\{6, 7, 9, 13\}$ with $d\geq \frac{1}{2}$.\\

By using a super-simple $(4, 2)-$DGDD of group type $3^7$ with $d\geq \frac{1}{2}$, (see Lemma 2), applying
Construction 1 and using a $TD(4, 5)$ as an input design we obtain a super-
simple $(4, 2)-$DGDD of group type $15^7$ with $d\geq \frac{1}{2}$. By Lemma \ref{3}, there exists a super-simple $ (15+1, 4, 2)$DD,
then by Contruction 2, we obtain a super-simple $(105+1, 4, 2)$DD with $d\geq \frac{1}{2}$.
\end{proof}
\begin{lemma}
There exists a super-simple $(v,4,2)$DD for $ v \in \{115,133\} $ with $d\geq \frac{1}{2}$.
\end{lemma}

\begin{proof}
For $ v=115 $, we delete one point from the last group of a $TD(5, 4)$ to obtain a $(\{4,5 \})-$
GDD of group type $ 4^{4}3^{1} $. By applying Construction \ref{1} and using  a super-simple $(4,2)-$DGDDs
of group type $ 6^{4} $ and $ 6^{5} $ with $d\geq \frac{1}{2}$ from Lemma \ref{5} and Lemma \ref{8}, 
as input designs,  
we get a super-simple $(4,2)-$DGDD of group type $ 24^{4} 18^{1}$,
with $d\geq \frac{1}{2}$. 
Since for each group G,  $ | G |+1\in \{19,25\}$, 
 by Lemma \ref{3} and Construction 2, there exsits a super-simple$ (115,4,2)$DD with $d\geq \frac{1}{2}$.\\
 
For $ v=133 $, we delete three points from the last group of a $TD(5,5)$ to obtain a $(\{4, 5\}, 1)-$GDD
of group type $5^{4}2^{1}$. Applying Construction \ref{1} and using  super-simple $(4,2)-$DGDDs of group
types $6^{ 4}$ and $6^{5}$ as input designs, we get a super-simple $(4,2)-$DGDD of
group type $30^{4}12^{1}$. Since for each group G, $ | G |+1\in \{13,31\}  $,
 by Lemma \ref{3} and Construction 2, there exsits a super-simple $ (133,4,2)$DD.
 \end{proof}

Let $N(m)$ be the number of mutually orthogonal latin squares of order $m$. The following lemmas are generalized forms of Lemmas 4.27 and 4.28 in \cite{chen1}.
\begin{lemma}\label{18}
Let $N(m) \geq 6$, ${a,b} \in [0,m]$. If there exists a super-simple $ (v,4,2)$DD
for $v \in \{3m + 1,3a + 1,3b + 1\}$ with $d\geq \frac{1}{2}$,
 then there exists a super-simple $ (18m + 3a + 3b + 1,4,2)$DD with $d\geq \frac{1}{2}$.
\end{lemma}

\begin{proof}
Delete $m - a$ and $m - b$ points from the last two groups of a $TD(8, m)$ to
obtain a $(\{6,7, 8\})-$GDD of group type $m^{6}a^{1}b^{1}$. Apply Construction \ref{1}
and use  super-simple $(4,2)-$DGDDs of group type $3^{k}$ for $k \in \{6,7,8,9\} $ as input designs.
We obtain a super-simple $ (4,2)-$DGDD of group type $(3m)^{6} (3a)^{1} (3b)^{1}$.
Since by Lemma \ref{1} there exists a super-simple $ (v,4,2)$DD for $v \in \{3m + 1,3a + 1,3b + 1\}$, 
 we conclude that there exists a super-simple$ (18m + 3a + 3b + 1,4,2)$DD.
This completes the proof.
\end{proof}
\begin{lemma}\label{19}
Let $N(m) \geq 7$, ${a,b,c} \in [0,m]$. If there exists a super-simple $ (v,4,2)$DD
for $v \in \{3m + 1,3a + 1,3b + 1,3c+1\}$ with $d\geq \frac{1}{2}$,
 then there exists a super-simple $ (18m + 3a + 3b +3c+ 1,4,2)$DD with $d\geq \frac{1}{2}$.
\end{lemma}
\begin{proof}The proof is similarly to the proof of Lemma \ref{18}.
\end{proof}
\begin{lemma}\label{21}

There exists a super-simple $(v,4,2)$DD for $ v \in M=\{136,139,$ \\ $277,280,283,292,295,298,301,409,412,415,
424,427,430,433,436,439,442,445,$ \\ $
448,454,457,460,463,478,481,496,499,553\} $ with $d\geq \frac{1}{2}$.
\end{lemma}

\begin{proof}
For $ v \in \{136,139\} $, apply Lemma \ref{18} with $ m=7 $ and $ a+b \in \{ 3,4 \} $.\\
For $ v \in M \setminus \{136,139\} $ apply Lemma \ref{19} with the parameters shown in following table.
We conclude that for  these values of $ v $ there exists a super-simple $ (v,4,2)$DD.\\

Let $[a,b]_3^1$ denotes the set $\{v\vert  \  v\equiv 1$ (mod $3), \ \  a\leq v\leq b\}$.\\

 \begin{tabular}{|l|l|l|}
\hline $v$ & $m$ & $a+b+c \subseteq [3,m] \cup \{0\}$ \\ 
\hline $[277,298]_{3}^{1}$ & 11 & $a+b+c \in$[26,33]$ \setminus $\{29,30\} \\ 
 301 & 16 & $ a+b+c= $4 \\ 
 $[409,433]_{3}^{1}$  $ \setminus $\{418,421\}  & 16 & $a+b+c \in$[40,48]$ \setminus $\{43,44\} \\ 
$[436,463]_{3}^{1}$ $ \setminus $\{451\}  &  19 &  $a+b+c \in$[31,40]$ \setminus $\{36\} \\ 
\{478,481\} & 19 &$a+b+c \in$ \{45,46\} \\ 
\{496,499\} & 19 &$a+b+c \in$ \{51,52\}\\ 
  553 &  19 & $a+b+c =$ 22 \\
\hline 
\end{tabular}\\ 

This completes the proof.
\end{proof}
\begin{lemma}\cite{abel}\label{10}
There exists  a $PBD(n, (6,7,8,9\})$ for all $n$ except $n\in [10,30],$ \\ $[32,41], [45,47], [93,95], [98,101], [137,139],
 [142,150], [152,155], \{160, 161, 166,$ \\ $167,185\}$.
\end{lemma}
\begin{corollary}\label{111}
There exists a $\{6,7,8,9\}-$GDD   with group type $ 5^{a} 6^{b} 7^{c} 8^{d}$,
 where  a,b,c,d  are nonnagative integers and $5a+6b+7c+8d=m$ for all $ m $, except $ m \in [9,29], [31,40], [44,46], [92,94], [97,100], [136,138],
 [141,149], [151,154], \{159, 160,$ \\ $165,166,184\} $. 
\end{corollary}
Now, we are in a position to conclude the main result.\\
\textit{{\bf Main Theorem.} For all   $ v\equiv 1$  (\textsl{mod 3} )  and $ v\geq 10 $, there exists a super-simple $ (v,4,2)$DD with $d\geq \frac{1}{2}$.}

\begin{proof}
By Corollary \ref{111},  there exists $\{6,7,8,9\}-$GDD of order $m$ and group type $5^a6^b7^c8^d$ where $a,b,c$ and $d$ are nonnegative integers and $m=5a+6b+7c+8d$ for values of $m$ satisfied in Corollary \ref{111}. We apply Construction 1 to this GDD by using a weight 3 to get a super-simple $(4,2)-$DGDD of type $15^a18^b21^c24^d$. Finally, by Lemma \ref{3} and Construction 2, we get a super-simple $(3m+1,4,2)$DD. For the remaining values of $m$, the existence of super-simple $(3m+1,4,2)$DD has been proved in Lemmas \ref{3} and \ref{11}-\ref{21}.
The proof is complete.
\end{proof}

\end{document}